\documentclass[12pt]{article}
\usepackage{amssymb,latexsym,amsmath,enumerate,verbatim,amsfonts,amsthm}
\usepackage[active]{srcltx}
\numberwithin{equation}{section}
\newtheorem{Theorem}{Theorem}[section]
\newtheorem{Definition}{Definition}[section]

\newtheorem{Lemma}{Lemma}[section]

\newtheorem{Question}{Question}

\makeatletter
 \def\@biblabel#1{#1.}
\makeatother

\newcommand{\uu}{{\bf u}}

\newcommand{\x}{{\bf x}}
\newcommand{\y}{{\bf y}}

\newcommand{\q}{{\bf q}}

\newcommand{\0}{{\bf 0}}

\begin{document}
\title{Eigenvalues and structured properties of P-tensors}

\author{Yisheng Song\thanks{School of Mathematics and Information Science  and Henan Engineering Laboratory for Big Data Statistical Analysis and Optimal Control,
 Henan Normal University, XinXiang HeNan,  P.R. China, 453007.
Email: songyisheng1@gmail.com. This author's work was supported by
the National Natural Science Foundation of P.R. China (Grant No.
11571905).  His work was partially done when he was
visiting The Hong Kong Polytechnic University.}, \quad Liqun Qi\thanks{ Department of Applied Mathematics, The Hong Kong Polytechnic University, Hung Hom,
Kowloon, Hong Kong. Email: maqilq@polyu.edu.hk. This author's work was supported by the Hong
Kong Research Grant Council (Grant No. PolyU 502111, 501212, 501913 and 15302114).}}

\date{\today}

 \maketitle

\begin{abstract}
\noindent  
 We define two new constants associated with real eigenvalues of a P-tensor. With the help of these two constants, in the case of P-tensors, we establish upper bounds of two important quantities, whose positivity is a necessary and sufficient condition for a general tensor to be a P-tensor.

\noindent {\bf Key words:}\hspace{2mm} P-tensor,  Complementarity problem,   Eigenvalues.  \vspace{3mm}

\noindent {\bf AMS subject classifications (2010):}\hspace{2mm}
47H15, 47H12, 34B10, 47A52, 47J10, 47H09, 15A48, 47H07.
  \vspace{3mm}

\end{abstract}


\section{Introduction}
\hspace{4mm}
Let $A = (a_{ij})$ be an $n \times n $ real matrix and $\q \in \mathbb{R}^n$.  Then the linear complementarity problem, denoted by LCP $(A, \q)$, is to find $ \x \in \mathbb{R}^n$ such that
$$	\x \geq \0, \q + A\x \geq \0, \mbox{ and }\x^\top (\q + A\x) = 0
\leqno{\bf LCP(A,\q)}$$
or to show that no such vector exists. It is well-known that the LCP$(A,\q)$ has wide and important applications in engineering and economics (Cottle, Pang and Stone \cite{CPS} and Han, Xiu and Qi \cite{HXQ}).

In past several decades, there have been a growing literature concerned with the error bounds for LCP $(A, \q)$. The error bounds for LCP $(A, \q)$ have been given in Chen and Xiang \cite{CX07,CX06}, Mathias and Pang \cite{MP90} for  P-matrix; Chen, Li, Wu, Vong \cite{CLWV} for MB-matrix; Dai \cite{D2011} for DB-matrix; Dai, Li, Lu \cite{DLL12,DLL13} for SB-matrix; Garc\'ia-Esnaola and Pe\~na \cite{GP09} for B-Matrix; Garc\'ia-Esnaola and Pe\~na \cite{GP12} for BS-Matrix; Garc\'ia-Esnaola and Pe\~na \cite{GP10}, Li and Zheng \cite{LZ14} for H-Matrix; Luo, Mangasarian, Ren, Solodov \cite{LMRS} for nondegenerate matrix. Recently, Sun and Wang \cite{SW13} studied the error bounds  for generalized linear complementarity problem under some
proper assumptions. The componentwise error bounds for LCP $(A, \q)$ was showed by Wang and Yuan \cite{WY11}.

The nonlinear complementarity problem, defined by a nonlinear function $F: \mathbb{R}^n
\to \mathbb{R}^n,$  denoted by NCP$(F)$, is to  find a vector $ \x \in \mathbb{R}^n$ such that
$$ \x \geq \0, F(\x) \geq \0, \mbox{ and }\x^\top F(\x) = 0, \leqno{\bf NCP(F)} $$
or to show that no such vector exists. The NCP was introduced by Cottle in his Ph.D. thesis in 1964. the study of NCP$(F)$ have a long history and wide
applications in mathematical sciences and applied sciences (Facchinei and Pang \cite{FaP}). We call  the NCP(F) the {\bf tensor complementarity problem}, denoted by TCP $( \mathcal{A},\q)$ iff  $F(x)=\q + \mathcal{A}\x^{m-1}$ in the NCP(F), i.e.,  finding $\x \in \mathbb{R}^n$ such that 	$$ \x \geq \0, \q + \mathcal{A}\x^{m-1} \geq \0, \mbox{ and }\x^\top (\q + \mathcal{A}\x^{m-1}) = 0\leqno{\bf TCP(\mathcal{A},\q)}$$
 or showing that no such vector exists, where $\mathcal{A} = (a_{i_1\cdots i_m})$ is a real $m$th order $n$-dimensional tensor (hypermatrix).

The TCP $(\mathcal{A},\q)$ is a natural extension of the LCP $(A,\q)$. It has some similar  properties to the LCP $(A,\q)$.  At the same time, the TCP $(\mathcal{A},\q)$, as a specially  structured  NCP(F), should have its particular and nice properties other than the general NCP(F). So how to obtain the nice properties and their applications of the TCP $( \mathcal{A},\q)$ will be very interesting by means of the  special   structure of  higher order tensors (hypermatrices).  Recently, the solution of TCP$( \mathcal{A},\q)$ and related problems have been well studied. For example, Che, Qi, Wei \cite{CQW} investigated the existence and uniqueness of solution of TCP $(\mathcal{A},\q)$ for some special tensors. Song and Qi \cite{SQ2015,SQ15} studied  the existence of solution of  the TCP $(\mathcal{A},\q)$ with the help of the structure of the tensor $\mathcal{A}$. Song and Yu \cite{SY15} showed the properties of solution set of  the TCP $(\mathcal{A},\q)$.  Luo, Qi and Xiu \cite{LQX} obtained the sparsest solutions to the TCP $(\mathcal{A},\q)$ for Z-tensors. Song and Qi \cite{SQ13}, Ling, He, Qi \cite{LHQ2015, LHQ15}, Chen, Yang, Ye \cite{CYY} studied the the tensor eigenvalue complementarity problem for higher order tensors.   See these papers and references therein.

In this paper, we introduce two new constants associated with real eigenvalues of P-tensors. With the help of these two constants, for a P-tensor $\mathcal{A}$, we establish upper bounds of two quantities $\alpha(F_{\mathcal{A}})$ and $\alpha(T_{\mathcal{A}})$.    These two quantities were defined for general tensors and even order tensors respectively by Song and Qi \cite{SQ-2015}.  It was shown there that an $m$-order $n$-dimensional tensor $\mathcal{A}$ is a P-(P$_0$-)tensor if and only if $\alpha(T_{\mathcal{A}})$ is positive (nonnegative), and when $m$ is even, $\mathcal{A}$ is a P-(P$_0$-)tensor if and only if $\alpha(F_{\mathcal{A}})$ is positive (nonnegative).

The rest of this article is organized as follows. In Section 2,  we will  give some definitions and basic conclusions, which will be used later on. In Section 3,  we will define two new constants associated with real eigenvalues of P-tensors and show their upper and lower bounds. In particular, for a P-tensor $\mathcal{A}$, we establish upper bounds of $\alpha(F_{\mathcal{A}})$ and $\alpha(T_{\mathcal{A}})$.

We briefly describe our notation. Denote $\mathbb{R}^n:=\{(x_1, x_2,\cdots, x_n)^T; x_i\in  \mathbb{R},  i\in I_n\}$ and  ${\mathbb{C}}^n:=\{(x_1,x_2,\cdots,x_n)^T;x_i\in {\mathbb{C}},  i \in I_n \}$, where $\mathbb{R}$ is the set of real numbers. and ${\mathbb{C}}$ is the set of complex numbers.  Denote $I_n := \{ 1,2, \cdots, n \}$. For any vector $\x \in
\mathbb{C}^n$, $\x^{[m-1]}$ is a vector in $\mathbb{C}^n$ with
its $i$th component defined as $x_i^{m-1}$ for $i \in I_n$, and $\x \in
\mathbb{R}^n$, $\x_+$ is a vector in $\mathbb{R}^n$ with $(\x_+)_i=x_i$ if $x_i\geq0$ and $(\x_+)_i=0$ if $x_i<0$  for $i \in I_n$.  We assume that $m \ge 2$ and $n \ge 1$.   We
use small letters $x, u, v, \alpha, \cdots$, for scalars, small bold
letters $\x, \y, \uu, \cdots$, for vectors, capital letters $A, B,
\cdots$, for matrices, calligraphic letters $\mathcal{A}, \mathcal{B}, \cdots$, for tensors.    We
denote the zero tensor in $T_{m, n}$ by $\mathcal{O}$.  Denote the set of all
real $m$th order $n$-dimensional tensors by $T_{m, n}$. We denote by $\mathcal{A}^J_r$ the principal sub-tensor of a tensor $\mathcal{A} \in T_{m, n}$ such that the entries of  $\mathcal{A}^J_r$ are indexed by $J \subset I_n$ with $|J|=r$ ($1 \le r\leq n$), and denote by $\x_J$ the $r$-dimensional sub-vector of a vector $\x \in \mathbb{R}^n$, with the components of $\x_J$ indexed by $J$.

\section{Preliminaries and basic facts}
\hspace{4mm}

In this section, we will  collect some basic definitions and facts, which will be used later on.

All the tensors discussed in this paper are real. An $m$-order $n$-dimensional tensor (hypermatrix) $\mathcal{A} = (a_{i_1\cdots i_m})$ is a multi-array of real entries $a_{i_1\cdots
	i_m}$, where $i_j \in I_n$ for $j \in I_m$. If the entries $a_{i_1\cdots i_m}$ are
invariant under any permutation of their indices, then $\mathcal{A}$ is
called a {\bf symmetric tensor}.

Let $\mathcal{A} = (a_{i_1\cdots i_m}) \in T_{m, n}$ and $\x \in\mathbb{R}^n$.   Then $\mathcal{A} \x^{m-1}$ is a vector in $\mathbb{R}^n$ with
its $i$th component as
$$\left(\mathcal{A} \x^{m-1}\right)_i: = \sum_{i_2, \cdots, i_m=1}^n a_{ii_2\cdots
	i_m}x_{i_2}\cdots x_{i_m}$$ for $i \in I_n$.   We now give the
definitions of P-tensors, which was introduced by Song and Qi \cite{SQ-2015}.
\begin{Definition} \label{d1}\em
	Let $\mathcal{A}  = (a_{i_1\cdots i_m}) \in T_{m, n}$.   We say that $A$ is \begin{itemize}
		\item[(i)]
	 a P$_0$ tensor iff for any nonzero vector $\x$ in $\mathbb{R}^n$, there
	exists $i \in I_n$ such that $x_i \not = 0$ and
	$$ x_i \left(\mathcal{A} \x^{m-1}\right)_i \ge 0;$$
	\item[(ii)] a P tensor iff for any nonzero vector $\x$ in $\mathbb{R}^n$,
	$$\max_{i\in I_n}x_i \left(\mathcal{A} \x^{m-1}\right)_i > 0.$$
	\end{itemize}
\end{Definition}
The concepts of tensor eigenvalues were introduced by Qi \cite{LQ1,LQ2} to the higher order symmetric tensors, and the existence of the eigenvalues and some applications were studied there. Lim \cite{LL} independently introduced real tensor eigenvalues and obtained some existence results using a variational approach.
\begin{Definition} \label{d2}\em
	Let $\mathcal{A}  = (a_{i_1\cdots i_m}) \in T_{m, n}$.  A number $\lambda \in \mathbb{C}$ is called \begin{itemize}
		\item[(i)] an {\bf eigenvalue} of $\mathcal{A}$ iff there is a nonzero vector $\x \in \mathbb{C}^n$ such that
	\begin{equation} \label{eig}
	\mathcal{A} \x^{m-1} = \lambda \x^{[m-1]},
	\end{equation}
	and $\x$  is called an {\bf eigenvector} of $\mathcal{A}$, associated with $\lambda$.  An eigenvalue $\lambda$ corresponding a real eigenvector $\x$ is real and is called an {\bf
		H-eigenvalue}, and $\x$ is called an {\bf H-eigenvector} of $\mathcal{A}$, respectively;\\
	\item[(ii)]  an {\bf E-eigenvalue} of $\mathcal{A}$ iff there is a nonzero vector $\x \in \mathbb{C}^n$ such that
	\begin{equation} \label{eig1}
	\mathcal{A} \x^{m-1} = \lambda \x,  \ \ \x^\top \x = 1,
	\end{equation}
 and $\x$ is
	called an {\bf E-eigenvector} of $\mathcal{A}$,  associated with $\lambda$.  An E-eigenvalue  $\lambda$ corresponding a real E-eigenvector $\x$ is real and is called a {\bf
		Z-eigenvalue}, and $\x$ is called a {\bf Z-eigenvector} of $\mathcal{A}$, respectively.\end{itemize}
\end{Definition}
The concept of principal sub-tensors was introduced and used in \cite{LQ1} for symmetric tensors.
\begin{Definition} \label{d3}\em Let $\mathcal{A}  = (a_{i_1\cdots i_m}) \in T_{m, n}$.
A tensor $\mathcal{C} \in T_{m, r}$  is called {\bf a principal sub-tensor}  of a tensor $\mathcal{A} = (a_{i_1\cdots i_m}) \in T_{m, n}$ ($1 \le r\leq n$) iff there is a set $J$ that composed of $r$ elements in $I_n$ such that
$$\mathcal{C} = (a_{i_1\cdots i_m}),\mbox{ for all } i_1, i_2, \cdots, i_m\in J.$$
Denote such a principal sub-tensor $\mathcal{C}$ by $\mathcal{A}_r^J$. \end{Definition}
The following is a basic conclusion in the study of P-tensors.
\begin{Lemma}\em{(Song and Qi \cite[Theorem 3.1 and 4.1]{SQ-2015})}\label{le:21}\
Let $\mathcal{A} \in T_{m, n}$ be a P-tensor. Then
\begin{itemize}
	\item[(i)] all principal diagonal entries of $\mathcal{A}$ are positive ($a_{ii\cdots i}>0$ for all $i\in I_n$);
	 \item[(ii)]  each principal sub-tensor of $\mathcal{A}$ is a P-tensor;
	\item[(iii)] all  H-eigenvalues of each principal sub-tensor of $\mathcal{A}$ are positive  when $m$ is even;
 \item[(iv)] all Z-eigenvalues of each principal sub-tensor of $\mathcal{A}$ are positive  when $m$ is even.
 \end{itemize}
 \end{Lemma}
 Recall that an operator $T : \mathbb{R}^n \to \mathbb{R}^n$ is called {\bf
 	positively homogeneous} iff $T(t\x)=tT(\x)$ for each $t>0$ and all
 $\x\in \mathbb{R}^n$.  For $\x \in \mathbb{R}^n$, it is known well that
 $$\|\x\|_\infty:=\max\{|x_i|;i \in I_n \}\mbox{ and } \|\x
 \|_2:=\left(\sum_{i=1}^n|x_i|^2\right)^{\frac12}$$ are two main norms
 defined on $\mathbb{R}^n$.  Then for a continuous, positively homogeneous
 operator $T : \mathbb{R}^n\to \mathbb{R}^n$,  it is obvious that
 $$\|T\|_\infty:=\max_{\|\x\|_\infty=1}\|T(\x)\|_\infty$$ is an
 operator norm of $T$ and
 $\|T(\x)\|_\infty\leq\|T\|_\infty\|\x\|_\infty$ for any $\x \in
 \mathbb{R}^n$.

Let $\mathcal{A} \in T_{m, n}$.  Define an operator $T_\mathcal{A} : \mathbb{R}^n \to \mathbb{R}^n$
by for any $\x \in \mathbb{R}^n$,
\begin{equation} \label{TA}
T_\mathcal{A} (\x) := \begin{cases}\|\x \|_2^{2-m}\mathcal{A} \x^{m-1},\ \x \neq\0\\
\0,\ \ \ \ \ \ \ \ \ \ \  \ \ \ \ \ \ \ \
\x =\0.\end{cases}
\end{equation}
When $m$ is even, define another operator $F_\mathcal{A} : \mathbb{R}^n \to \mathbb{R}^n$ by for any $\x \in \mathbb{R}^n$,
\begin{equation} \label{FA}
F_\mathcal{A} (\x) :=\left(\mathcal{A}  \x^{m-1}\right)^{\left[\frac1{m-1}\right]}.
\end{equation}
 Clearly, both
 $F_\mathcal{A}$ and $T_\mathcal{A}$ are continuous and positively homogeneous. The
 following upper bounds of the operator norm were established  by
 Song and Qi \cite{SQ}.
 \begin{Lemma} \label{le:22} \em (Song and Qi \cite[Theorem 4.3]{SQ})  Let $\mathcal{A} = (a_{i_1\cdots i_m}) \in T_{m, n}$. Then
 	\begin{itemize}
 		\item[(i)] $\|T_\mathcal{A}\|_{\infty}\leq \max\limits_{i\in I_n}\left(\sum\limits_{i_2,\cdots,i_m=1}^{n}|a_{ii_2\cdots i_m}|\right)$;
 		\item[(ii)] $\|F_\mathcal{A}\|_{\infty}\leq\max\limits_{i\in I_n}\left(\sum\limits_{i_2,\cdots,i_m=1}^{n}|a_{ii_2\cdots i_m}|\right)^{\frac{1}{m-1}}$, when $m$ is even.
 	\end{itemize} \end{Lemma}

 Recently, Song and Qi \cite{SQ-2015} defined two quantities for a P$_0$-tensor $\mathcal{A}$ with the help of the above two operators.
\begin{equation} \label{Talpha}
\alpha(T_\mathcal{A}):=\min_{\|\x \|_\infty=1}\max_{i \in I_n}x_i(T_\mathcal{A}(\x))_i
\end{equation}
for any $m$, and
\begin{equation} \label{Falpha}
\alpha(F_\mathcal{A}):=\min_{\|\x \|_\infty=1}\max_{i \in I_n}x_i(F_\mathcal{A}(\x))_i
\end{equation}
 when $m$ is even.

 The monotonicity and boundedness of two constants $\alpha(T_\mathcal{A})$ and $\alpha(F_\mathcal{A})$ for a high order tensor $\mathcal{A}$ are showed by Song and Qi \cite{SQ-2015}.
\begin{Lemma}\label{le:23}\em (Song and Qi \cite[Theorem 4.3]{SQ-2015}) Let $\mathcal{A} = (a_{i_1\cdots i_m})$ be a P$_0$ tensor in $T_{m, n}$. Then
 \begin{itemize}
\item[(i)] $\alpha(T_\mathcal{A})\leq\alpha(T_{\mathcal{A}^J_r})$ for all principal sub-tensors $\mathcal{A}^J_r$;
\item[(ii)]  $\alpha(F_\mathcal{A})\leq\alpha(F_{\mathcal{A}^J_r})$ for all principal sub-tensors $\mathcal{A}^J_r$, when $m$ is even;
\item[(iii)]  $\alpha(T_\mathcal{A})\leq \max\limits_{i\in I_n}\left(\sum\limits_{i_2,\cdots,i_m=1}^{n}|a_{ii_2\cdots i_m}|\right)$;
\item[(iv)]  $\alpha(F_\mathcal{A})\leq \max\limits_{i\in I_n}\left(\sum\limits_{i_2,\cdots,i_m=1}^{n}|a_{ii_2\cdots i_m}|\right)^{\frac1{m-1}}$, when $m$ is even.
\end{itemize}
  \end{Lemma}

The necessary  and sufficient conditions for P-tensor based upon $\alpha(F_\mathcal{A})$ and $\alpha(T_\mathcal{A})$ are obtained by Song and Qi \cite{SQ-2015}.
 \begin{Lemma}\label{le:24}\em (Song and Qi \cite[Theorem 4.4]{SQ-2015}) Let $\mathcal{A} \in T_{m, n}$.  Then
 \begin{itemize}
\item[(i)] $\mathcal{A}$ is a P-tensor if and only if $\alpha(T_\mathcal{A})>0$;
\item[(ii)] when $m$ is even, $\mathcal{A}$ is a P-tensor if and only if $\alpha(F_\mathcal{A})>0$.
\end{itemize}
  \end{Lemma}
 The following conclusions about the solution of TCP$(\mathcal{A},\q)$ with P-tensor $\mathcal{A}$ are obtained by Song and Qi \cite{SQ2015,SQ15}
\begin{Lemma}\label{le:25}\em(Song and Qi \cite[Corollary 3.3, Theorem 3.4]{SQ2015} and \cite[Theorem 3.1]{SQ15}) Let $\mathcal{A} \in T_{m, n}$ be a P-tensor. Then the TCP$(\mathcal{A},\q)$ has a solution have a solution for all $\q\in \mathbb{R}^n$, and  has only zero vector solution for $\q\geq\0$.
\end{Lemma}


\section{Upper bounds of $\alpha(F_{\mathcal{A}})$ and $\alpha(T_{\mathcal{A}})$}
\hspace{4mm}

 The quantities $\alpha(T_\mathcal{A})$ and $\alpha(F_\mathcal{A})$ play a fundamental role  in the error bound analysis of TCP($\q, \mathcal{A}$).  The two quantities are in general not easy to compute. However, it is easy to derive some upper bounds for them when $\mathcal{A}$ is a P-tensor. Recently, Song and Qi \cite{SQ-2015} obtained the monotonicity and boundedness of two constants $\alpha(T_\mathcal{A})$ and $\alpha(F_\mathcal{A})$ for a P-tensor $\mathcal{A}$.  In this section, we will establish some smaller upper bounds. For this purpose, we introduce two quantities about a P-tensor $\mathcal{A}$:
 \begin{equation}\label{eq:41}
\delta_H(\mathcal{A}) := \min\{\lambda_H(\mathcal{A}_r^J); J\subset I_n, r\in I_n\},
\end{equation}
where  $\lambda_H(\mathcal{A})$ denotes the smallest of H-eigenvalues (if any exists) of a P-tensor $\mathcal{A}$;
 \begin{equation}\label{eq:42}
\delta_Z(\mathcal{A}) := \min\{\lambda_Z(\mathcal{A}_r^J); J\subset I_n, r\in I_n\},
\end{equation}
where  $\lambda_Z(\mathcal{A})$ denotes the smallest of Z-eigenvalues (if any exists) of a P-tensor $\mathcal{A}$. The above two minimum ranges over those principal sub-tensors of $\mathcal{A}$ which indeed have H-eigenvalues (Z-eigenvalues).

It follows from Lemma \ref{le:21} that all principal diagonal ertries of $\mathcal{A}$ are positive and all  H-(Z-)eigenvalues of each principal sub-tensor of $\mathcal{A}$ are positive when $m$ is even. So $\delta_H(\mathcal{A})$ and $\delta_Z(\mathcal{A})$ are well defined, finite and positive when $m$ is even.  Now we give some upper bounds of $\alpha(F_{\mathcal{A}})$ and $\alpha(T_{\mathcal{A}})$ using the quantities $\delta_H(\mathcal{A})$ and $\delta_Z(\mathcal{A})$.

\begin{Theorem}\label{th:41}\em Let $\mathcal{A} \in T_{m, n}$ ($m\geq2$) be a P-tensor, and let $m$ be an even number.  Then
	\begin{equation}\label{eq:43}
\alpha(F_{\mathcal{A}})\leq (\delta_H(\mathcal{A}))^{\frac1{m-1}}\leq (\min_{i\in I_n}a_{ii\cdots i})^{\frac1{m-1}}.
	\end{equation}
\end{Theorem}
\begin{proof} It follows from Lemma \ref{le:21} (i) that $$a_{ii\cdots i}>0\mbox{ for all }i\in I_n.$$ Since $\mathcal{A}^J_1=(a_{ii\cdots i})$ ($J=\{i\}$) is m-order 1-dimensional principal sub-tensor of $\mathcal{A}$,  then $a_{ii\cdots i}$ is a H-eigenvalue of $\mathcal{A}^J_1$ for all $i\in I_n$, and hence
$$\delta_H(\mathcal{A})\leq \min_{i\in I_n}a_{ii\cdots i}.$$

Next we show the left-hand inequality. Let $\delta=\delta_H(\mathcal{A})$ and $\mathcal{B}=\mathcal{A}-\delta\mathcal{I}$, where $\mathcal{I}$ is unit tensor.  Then it follows from the definition of $\delta_H(\mathcal{A})$ that $\delta$ is  a H-eigenvalue of a principal sub-tensor $\mathcal{A}^J_r$ of $\mathcal{A}$. Then there exists $\x^*\in \mathbb{R}^r\setminus\{\0\}$ such that  $$\left(\mathcal{A}^J_r-\delta\mathcal{I}^J_r\right)(\x^*)^{m-1}=\mathcal{A}^J_r(\x^*)^{m-1}-\delta(\x^*)^{[m-1]}=\0.$$
 So the principal sub-tensor $\mathcal{B}^J_r=\mathcal{A}^J_r-\delta\mathcal{I}^J_r$ of $\mathcal{B}$ is not a P-tensor. Thus  it follows from Lemma \ref{le:21} (ii) that $\mathcal{B}=\mathcal{A}-\delta\mathcal{I}$ is not a P-tensor.
Consequently, there exists a vector $\y$ with $\|\y\|_\infty=1$ such that
$$\max_{i\in I_n}y_i(\mathcal{B}\y^{m-1})_i=\max_{i\in I_n}\left(y_i(\mathcal{A}\y^{m-1})_i-\delta y_i^{m}\right)\leq 0.$$
So, we have $$y_l(\mathcal{A}\y^{m-1})_l-\delta y_l^{m}\leq\max_{i\in I_n}\left(y_i(\mathcal{A}\y^{m-1})_i-\delta y_i^{m}\right)\leq 0 \mbox{ for all }l\in I_n,$$
which implies that for some $j\in I_n$,
$$\max_{i\in I_n}y_i(\mathcal{A}\y^{m-1})_i=y_j(\mathcal{A}\y^{m-1})_j\leq\delta y_j^{m}\leq \delta \|\y\|_\infty^m=\delta.$$
	It follows from the definition (Equation (\ref{Falpha})) of $\alpha(F_{\mathcal{A}})$ that
	\begin{align}
	\alpha(F_{\mathcal{A}})\leq& \max_{i\in I_n}y_i(F_{\mathcal{A}}(\y))_i=\max_{i\in I_n}y_i(\mathcal{A} \y^{m-1})_i^{\frac1{m-1}}\nonumber\\
	=&\max_{i\in I_n}y_i^{\frac{m-2}{m-1}}(y_i(\mathcal{A} \y^{m-1})_i)^{\frac1{m-1}}\nonumber\\
	\leq& \|\y\|_\infty^{\frac{m-2}{m-1}}\max_{i\in I_n}(y_i(\mathcal{A} \y^{m-1})_i)^{\frac1{m-1}}\nonumber\\
	\leq& \delta^{\frac1{m-1}}.\nonumber
	\end{align}
	The desired inequality follows.
\end{proof}
\begin{Theorem}\label{th:42}\em Let $\mathcal{A} \in T_{m, n}$ ($m\geq2$) be a P-tensor, and let $m$ be an even number.  Then
	\begin{equation}\label{eq:44}
\alpha(T_{\mathcal{A}})\leq \delta_Z(\mathcal{A})\leq \min_{i\in I_n}a_{ii\cdots i}.
	\end{equation}
\end{Theorem}
\begin{proof} Using similar  proof technique of Theorem \ref{th:41}, we have the right-hand inequality holds. Next we show the left-hand inequality. Let $\delta=\delta_Z(\mathcal{A})$ and $\mathcal{B}=\mathcal{A}-\delta\mathcal{E}$, where $\mathcal{E}=I^{\frac{m}2}_2$ and $I_2$ is $n\times n$ unit matrix ($\mathcal{E}\x^{m-1}=\|\x\|^{m-2}_2\x$, see Chang, Pearson, Zhang \cite{CPZ09}).  Then it follows from the definition of $\delta_Z(\mathcal{A})$ that $\delta$ is  a Z-eigenvalue of a principal sub-tensor $\mathcal{A}^J_r$ of $\mathcal{A}$. Then there exists $\x^*\in \mathbb{R}^r\setminus\{\0\}$ such that $(\x^*)^\top\x^*=1$ and $$\left(\mathcal{A}^J_r-\delta\mathcal{E}^J_r\right)(\x^*)^{m-1}=\mathcal{A}^J_r(\x^*)^{m-1}-\delta\x^*=\0.$$ So the principal sub-tensor $\mathcal{B}^J_r=\mathcal{A}^J_r-\delta\mathcal{E}^J_r$ of $\mathcal{B}$ is not a P-tensor. Thus  it follows from Lemma \ref{le:21} (ii) that $\mathcal{B}=\mathcal{A}-\delta\mathcal{E}$ is not a P-tensor.
Consequently, there exists a vector $\y$ with $\|\y\|_\infty=1$ such that
$$\max_{i\in I_n}y_i(\mathcal{B}\y^{m-1})_i=\max_{i\in I_n}\left(y_i(\mathcal{A}\y^{m-1})_i-\delta\|\y\|_2^{m-2} y_i^2\right)\leq 0.$$
So, we have $$y_l(\mathcal{A}\y^{m-1})_l-\delta \|\y\|_2^{m-2} y_l^2\leq\max_{i\in I_n}\left(y_i(\mathcal{A}\y^{m-1})_i-\delta \|\y\|_2^{m-2} y_i^2\right)\leq 0\mbox{ for all }l\in I_n,$$
which implies that for some $j\in I_n$,
$$\max_{i\in I_n}y_i(\mathcal{A}\y^{m-1})_i=y_j(\mathcal{A}\y^{m-1})_j\leq\delta \|\y\|^{m-2}_2 y^2_j\leq \delta \|\y\|^{m-2}_2\|\y\|_\infty^2=\|\y\|^{m-2}_2\delta.$$
It follows from the definition (Equation (\ref{Talpha})) of $\alpha(T_{\mathcal{A}})$ that
	\begin{align}
 \alpha(T_{\mathcal{A}})\leq& \max\limits_{i\in I_n}y_i(T_{\mathcal{A}}(\y))_i=\max_{i\in I_n}y_i(\|\y\|_2^{2-m}\mathcal{A} \y^{m-1})_i\nonumber\\
 =&\|\y\|_2^{2-m}\max_{i\in I_n}y_i(\mathcal{A} \y^{m-1})_i\nonumber\\
 \leq& \|\y\|_2^{2-m}\|\y\|^{m-2}_2\delta\nonumber\\
 =& \delta\nonumber.\nonumber
 \end{align}
	The desired inequality follows.
\end{proof}
\begin{Question}\label{Q1}\em It is known from Lemma \ref{le:24} and Theorem \ref{th:41},\ref{th:42} that for  a P-tensor $\mathcal{A}$,
$$(\min_{i\in I_n}a_{ii\cdots i})^{\frac1{m-1}}\geq \alpha(F_{\mathcal{A}})>0\mbox{ and } \min_{i\in I_n}a_{ii\cdots i}\geq\alpha(T_{\mathcal{A}})>0.$$
Then we have the following questions for further research. \begin{itemize}
  \item[(i)]  Do two constants $\alpha(F_{\mathcal{A}})$ and $\alpha(T_{\mathcal{A}})$ have a positive lower bound?
  \item[(ii)]  Are the above upper bounds is the smallest?
  \end{itemize}
\end{Question}


\section{Conclusions}
\hspace{4mm}
In this paper,
We introduce two quantities about a P-tensor $\mathcal{A}$ by means of H- and Z-eigenvalues of real tensors:
$$\delta_H(\mathcal{A}) := \min\{\lambda_H(\mathcal{A}_r^J); J\subset I_n, r\in I_n\},$$ $$
\delta_Z(\mathcal{A}) := \min\{\lambda_Z(\mathcal{A}_r^J); J\subset I_n, r\in I_n\}.$$
The upper bounds are obtained, which only depend on the diagonal entries of tensor.
  \begin{itemize}
  \item[(iii)]  $\alpha(F_{\mathcal{A}})\leq (\delta_H(\mathcal{A}))^{\frac1{m-1}}\leq (\min\limits_{i\in I_n}a_{ii\cdots i})^{\frac1{m-1}} $ when $m$ is even.
  \item[(iv)] $\alpha(T_{\mathcal{A}})\leq \delta_Z(\mathcal{A})\leq \min\limits_{i\in I_n}a_{ii\cdots i}$ when $m$ is even.
  \end{itemize}



\end{document}